\documentclass{amsart}
\usepackage[latin1]{inputenc}

\usepackage{amssymb,amsmath}
\usepackage{amsthm}
\usepackage{paralist}
\usepackage{pictexwd,dcpic}
\usepackage{url}

\theoremstyle{plain}
\newtheorem{theorem}{Theorem}[section]
\newtheorem{proposition}[theorem]{Proposition}
\newtheorem{lemma}[theorem]{Lemma}
\newtheorem{corollary}[theorem]{Corollary}

\theoremstyle{definition}

\DeclareMathOperator{\range}{Im} 
\DeclareMathOperator{\Alt}{Alt} 
\newcommand{\kHOM}{\ensuremath{k\text{-}\mathbf{HOM}}}
\newcommand{\kCORE}{\ensuremath{k\text{-}\mathbf{CORE}}}
\newcommand{\HOM}{\ensuremath{\mathbf{HOM}}}
\newcommand{\CORE}{\ensuremath{\mathbf{CORE}}}
\newcommand{\kQdHOM}{\ensuremath{k\text{-}(Q,d)\text{-}\mathbf{HOM}}}
\newcommand{\HHOM}{\ensuremath{H\text{-}\mathbf{HOM}}}

\newcommand{\kommentti}[1]{}

\def\dotcup{\mathop{\mathaccent\cdot\cup}} 
\newlength{\bigcupwidth}
\DeclareMathOperator*{\bigdotcup}{\makebox[\bigcupwidth][c]{\makebox[0mm][c]{$\cdot$}\makebox[0mm][c]{$\bigcup$}}}

\begin{document}

\title{On the homomorphism order of labeled posets}

\author{Léonard Kwuida}
\address[L. Kwuida]{Zurich University of Applied Sciences \\ School of Engineering \\ Technikumstrasse 9 \\ CH-8401 Winterthur\\ Switzerland.}
\email{leonard.kwuida@zhaw.ch}

\author{Erkko Lehtonen}
\address[E. Lehtonen]{University of Luxembourg \\ Faculty of Science, Technology and Communication \\ 6, rue Richard Coudenhove-Kalergi \\ L-1359 Luxembourg \\ Luxembourg}
\email{erkko.lehtonen@uni.lu}
\thanks{This research was partially supported by the Academy of Finland, grant \#{}120307.}

\begin{abstract}
Partially ordered sets labeled with $k$ labels ($k$-posets) and their homomorphisms are examined. We give a representation of directed graphs by $k$-posets; this provides a new proof of the universality of the homomorphism order of $k$-posets. This universal order is a distributive lattice. We investigate some other properties, namely the infinite distributivity, the computation of infinite suprema and infima, and the complexity of certain decision problems involving the homomorphism order of $k$-posets. Sublattices are also examined.
\end{abstract}

\maketitle

\section{Introduction}

A partially ordered set labeled with $k$ labels ($k$-poset), also known as a partially ordered multiset (pomset) or a partial word, is an object $(P;{\leq},c)$, where $(P;{\leq})$ is a partially ordered set and $c$ is a function that assigns to each element of $P$ a label from the set $\{0, 1, \ldots, k-1\}$. A homomorphism between $k$-posets is a mapping $h \colon (P;{\leq},c) \to (P';{\leq}',c')$ that preserves both order and labels. A quasiorder, called the homomorphism quasiorder, can be defined on the set of all $k$-posets as follows: $(P;{\leq},c) \leq (P';{\leq}',c')$ if and only if there is a homomorphism of $(P;{\leq},c)$ to $(P';{\leq}',c')$.

Labeled posets have been used as a model of parallel processes (see Pratt \cite{Pratt}), and they can be viewed as a generalization of strings. Algebraic properties of labeled posets have been studied by Grabowski \cite{Grabowski}, Gischer \cite{Gischer}, Bloom and Ésik \cite{BE}, and Rensink \cite{Rensink}. Homomorphisms of $k$-posets were studied in the context of Boolean hierarchies of partitions by Kosub \cite{Kosub}, Kosub and Wagner \cite{KW}, and Selivanov \cite{Selivanov}. Kuske \cite{Kuske} and Kudinov and Selivanov \cite{KS} studied the undecidability of the first-order theory of the homomorphism quasiorder of $k$-posets. The second author applied $k$-posets to analyse substitution instances of operations on finite sets when the inner functions are monotone functions (with respect to some fixed partial order on the base set) \cite{ULM} and showed that for $k \geq 2$, the homomorphism order of finite $k$-posets is a distributive lattice which is universal in the sense that it admits an embedding of every countable poset~\cite{Universal}. But these are not complete lattices.

The current paper continues the investigation of some properties and sublattices of the homomorphism order of $k$-posets. We establish a representation of directed graphs by $k$-posets, which gives rise to a new proof of the universality of the homomorphism order of $k$-posets and enables us to study the complexity of certain decision problems related to $k$-posets. We are also interested in computing with infinite suprema and infima. In particular we examine join-infinite distributivity (JID) and its dual, meet-infinite distributivity (MID); these are special cases of complete infinite distributivity (CID). These properties are defined by the following identities:
\begin{description}
\item[JID] $x \wedge \bigvee\{x_i \mid i \in I\} = \bigvee\{x \wedge x_i \mid i \in I\}$,
\item[MID] $x \vee \bigwedge\{x_i \mid i \in I\} = \bigwedge\{x \vee x_i \mid i \in I\}$,
\item[CID] $\bigwedge \bigl\{ \bigvee \{a_{ij} \mid j \in J\} \mid i \in I \bigr\} = \bigvee \bigl\{ \bigwedge\{a_{i \varphi(i)} \mid i \in I\} \bigm\vert \varphi \colon I \to J \bigr\}$,
\end{description}
for $I,J\neq\emptyset$.

\section{Labeled posets and homomorphisms}
\label{sec:kposets}

For a positive natural number $k$, a \emph{partially ordered set labeled with $k$ labels} (\emph{$k$-poset}) is an object $(P; {\leq}, c)$, where $(P; {\leq})$ is a partially ordered set and $c \colon P \to \{0, 1, \ldots, k-1\}$ is a \emph{labeling} function. A \emph{labeled poset} is a $k$-poset for some $k$. Every subset $P'$ of a $k$-poset $(P;{\leq},c)$ may be considered as a $k$-poset $(P';{\leq}|_{P'}, c|_{P'})$, called a \emph{$k$-subposet} of $(P;{\leq},c)$. We often simplify these notations and write $(P,c)$ or $P$ instead of $(P;{\leq},c)$, and we simply write $c$ for the restriction $c|_S$ of $c$ to any subset $S$ of its domain. If the underlying poset of a $k$-poset is a lattice, chain, tree, or forest, then we refer to \emph{$k$-lattices,} \emph{$k$-chains,} \emph{$k$-trees,} \emph{$k$-forests,} and so on. For $k \leq l$, every $k$-poset is also an $l$-poset. Finite $k$-posets can be represented by Hasse diagrams with numbers designating the labels assigned to each element; see the various figures of this paper. For general background on partially ordered sets and lattices, see any textbook on the subject, e.g., \cite{DP,Gratzer}.

A $k$-chain $a_1 < a_2 < \dots < a_n$ with labeling $c$ is \emph{alternating,} if $c(a_i) \neq c(a_{i+1})$ for all $1 \leq i \leq n-1$. The \emph{alternation number} of a $k$-poset $(P,c)$, denoted $\Alt(P,c)$, is the cardinality of the longest alternating $k$-chain that is a $k$-subposet of $(P,c)$.

We will adopt much of the terminology used for graphs and their homomorphisms (see \cite{HN}). (Recall that a graph \emph{homomorphism} $h \colon G \to G'$ is an edge-preserving mapping between the vertex sets of graphs $G$ and $G'$. A \emph{core} is a graph that does not admit a homomorphism to any proper subgraph of itself.) Let $(P,c)$ and $(P',c')$ be $k$-posets. A mapping $h \colon P \to P'$ that preserves both ordering and labels (i.e., $h(x) \leq h(y)$ in $P'$ whenever $x \leq y$ in $P$, and $c = c' \circ h$) is called a \emph{homomorphism} of $(P,c)$ to $(P',c')$ and denoted $h \colon (P,c) \to (P',c')$. The composition of homomorphisms is again a homomorphism. An \emph{endomorphism} of $(P,c)$ is a homomorphism $h \colon (P,c) \to (P,c)$. If a homomorphism $h \colon (P,c) \to (P',c')$ is bijective and the inverse of $h$ is a homomorphism of $(P',c')$ to $(P,c)$, then $h$ is called an \emph{isomorphism,} and $(P,c)$ and $(P',c')$ are said to be \emph{isomorphic.}

We denote by $\mathcal{P}_k$ and $\mathcal{L}_k$ the classes of all finite $k$-posets and $k$-lattices, respectively. We define a quasiorder $\leq$ on $\mathcal{P}_k$ as follows: $(P,c) \leq (P',c')$ if and only if there is a homomorphism of $(P,c)$ to $(P',c')$. Denote by $\equiv$ the equivalence relation on $\mathcal{P}_k$ induced by $\leq$. If $(P,c) \equiv (P',c')$, we say that $(P,c)$ and $(P',c')$ are \emph{homomorphically equivalent.} We denote by $\tilde{\mathcal{P}}_k$ the quotient set $\mathcal{P}_k / {\equiv}$, and the partial order on $\tilde{\mathcal{P}}_k$ induced by the homomorphism quasiorder $\leq$ is also denoted by $\leq$. The quasiorder $\leq$ and the equivalence relation $\equiv$ can be restricted to $\mathcal{L}_k$, and we denote by $\tilde{\mathcal{L}}_k$ the quotient set $\mathcal{L}_k / {\equiv}$.

The homomorphic equivalence class of $(P,c) \in \mathcal{P}_k$ is denoted by $[(P,c)] = \{(P',c') \in \mathcal{P}_k \mid (P,c) \equiv (P',c')\}$. We tend to identify the $\equiv$-classes by their representatives; that is, whenever we say that $(P,c)$ is an element of $\tilde{\mathcal{P}}_k$, it is to be understood as referring to the $\equiv$-class $[(P,c)]$.

A $k$-poset $(P,c)$ is a \emph{core,} if all endomorphisms of $(P,c)$ are surjective (equivalently, if $(P,c)$ is not homomorphically equivalent to any $k$-poset of smaller cardinality). Every $k$-poset is homomorphically equivalent to a core. Isomorphic $k$-posets are homomorphically equivalent by definition. Homomorphically equivalent $k$-posets are not necessarily isomorphic, but homomorphically equivalent cores are isomorphic. Thus we can take the cores as representatives of the homomorphic equivalence classes, and the restriction of the quasiorder $\leq$ on $\mathcal{P}_k$ to the set of cores is isomorphic to $(\tilde{\mathcal{P}}_k, {\leq})$.

Two elements $a$ and $b$ of a poset $P$ are \emph{connected,} if there exists a sequence $a_1, \dotsc, a_n$ of elements of $P$ such that $a_1 = a$, $a_n = b$, and for all $1 \leq i \leq n-1$ either $a_i \leq a_{i+1}$ or $a_i \geq a_{i+1}$. A nonempty poset is \emph{connected} if all pairs of its elements are connected. A \emph{connected component} of a poset $P$ is a subposet $C \subseteq P$ that is connected and such that for every $x \in P \setminus C$ the subposet $C \cup \{x\}$ is not connected. It is easy to verify that all homomorphic images of a connected poset are connected. A $k$-poset is a core if and only if all its connected components are cores and pairwise incomparable under $\leq$.

\section{Representation of directed graphs by $k$-posets}
\label{sec:representation}

Let $G = (V,E)$ be a directed graph. We associate with $G$ a $2$-poset $P_G := (P; {\leq}, c)$, where $P := (V \cup E) \times \{0, 1\}$, and $c(a,b) = b$ for all $a \in V \cup E$, $b \in \{0,1\}$, and the covering relations of $\leq$ are exactly the following:
\begin{compactitem}
\item $(a,0) < (a,1)$ for all $a \in V$,
\item $(a,1) < (a,0)$ for all $a \in E$,
\item for each edge $(u,v) \in E$, $(u,0) < ((u,v),0)$, $((u,v),1) < (v,1)$.
\end{compactitem}
It is clear from the construction that if $G$ is a subgraph of $H$, then $P_G$ is a $k$-subposet of $P_H$.
See Figure \ref{fig:PG} for an example of a directed graph and its representation by a $2$-poset.

\begin{figure}
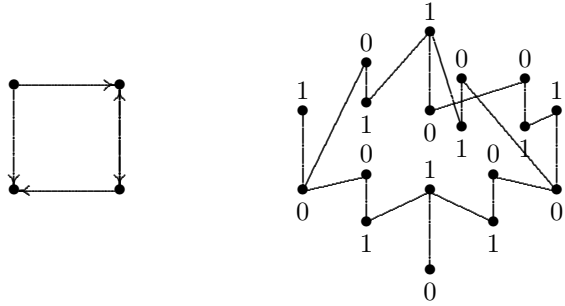

\[
\begindc{\digraph}[40]
\obj(0,1)[a]{}
\obj(1,1)[b]{}
\obj(1,0)[c]{}
\obj(0,0)[d]{}
\mor{a}{b}{}
\mor{a}{d}{}
\mor{c}{d}{}
\mor{b}{c}{}
\mor{c}{b}{}
\enddc
\qquad\qquad\qquad
\begindc{\undigraph}[3]
\obj(0,0)[a0]{0}[\south]
\obj(0,10)[a1]{1}[\north]
\obj(16,10)[b0]{0}[\south]
\obj(16,20)[b1]{1}[\north]
\obj(32,0)[c0]{0}[\south]
\obj(32,10)[c1]{1}[\north]
\obj(16,-10)[d0]{0}[\south]
\obj(16,0)[d1]{1}[\north]
\obj(8,11)[ab1]{1}[\south]
\obj(8,16)[ab0]{0}[\north]
\obj(8,-4)[ad1]{1}[\south]
\obj(8,2)[ad0]{0}[\north]
\obj(24,-4)[cd1]{1}[\south]
\obj(24,2)[cd0]{0}[\north]
\obj(20,8)[cb1]{1}[\south]
\obj(20,14)[cb0]{0}[\north]
\obj(28,8)[bc1]{1}[\south]
\obj(28,14)[bc0]{0}[\north]
\mor{a0}{a1}{}
\mor{b0}{b1}{}
\mor{c0}{c1}{}
\mor{d0}{d1}{}
\mor{a0}{ab0}{}
\mor{ab0}{ab1}{}
\mor{ab1}{b1}{}
\mor{a0}{ad0}{}
\mor{ad0}{ad1}{}
\mor{ad1}{d1}{}
\mor{c0}{cd0}{}
\mor{cd0}{cd1}{}
\mor{cd1}{d1}{}
\mor{b0}{bc0}{}
\mor{bc0}{bc1}{}
\mor{bc1}{c1}{}
\mor{c0}{cb0}{}
\mor{cb0}{cb1}{}
\mor{cb1}{b1}{}
\enddc
\]
\caption{Directed graph $G$ and its representation by a $2$-poset $P_G$.}
\label{fig:PG}
\end{figure}

\begin{proposition}
\label{GHhomom}
Let $G$ and $H$ be directed graphs. Then $G$ is homomorphic to $H$ if and only if $P_G$ is homomorphic to $P_H$.
\end{proposition}
\begin{proof}
Let $h \colon G \to H$ be a graph homomorphism. Then the mapping $g \colon P_G \to P_H$ defined as $g(v,b) = (h(v),b)$ for all $v \in V(G)$, $b \in \{0,1\}$; $g((u,v),b) = ((h(u),h(v)),b)$ for all $(u,v) \in E(G)$, $b \in \{0,1\}$, is easily seen to be a homomorphism. Clearly $g$ preserves the labels, and in order to show that $g(x) \leq g(y)$ in $P_H$ whenever $x \leq y$ in $P_G$ we have four cases to consider; recall that if $(u,v) \in E(G)$, then $(h(u),h(v)) \in E(H)$.
\begin{compactitem}
\item If $x = (u,0)$, $y = (u,1)$ where $u \in V(G)$, then $g(x) = g(u,0) = (h(u),0) < (h(u),1) = g(u,1) = g(y)$.
\item If $x = ((u,v),1)$, $y = ((u,v),0)$ where $u,v\in V(G)$ and $(u,v) \in E(G)$, then $g(x) = g((u,v),1) = ((h(u),h(v)),1) < ((h(u),h(v)),0) = g((u,v),0) = g(y)$.
\item If $x = (u,0)$, $y = ((u,v),0)$ where $u,v\in V(G)$ and $(u,v) \in E(G)$, then $g(x) = g(u,0) = (h(u),0) < ((h(u),h(v)),0) = g((u,v),0) = g(y)$.
\item If $x = ((u,v),1)$, $y = (v,1)$ where $u,v\in V(G)$ and $(u,v) \in E(G)$, then $g(x) = g((u,v),1) = ((h(u),h(v)),1) < (h(v),1) = g(v,1) = g(y)$.
\end{compactitem}

Assume then that $g \colon P_G \to P_H$ is a homomorphism. Since alternating chains must be mapped to isomorphic alternating chains by homomorphisms, we have that there are mappings $h \colon V(G) \to V(H)$, $e \colon E(G) \to E(H)$ such that $g(v,b) = (h(v),b)$ and $g((u,v),b) = (e(u,v),b)$ for all $v \in V(G)$, $(u,v) \in E(G)$, $b \in \{0,1\}$. Furthermore, the comparabilities $(u,0) < ((u,v),0)$, $((u,v),1) < (v,1)$ in $P_G$ must be preserved by $g$ for all edges $(u,v) \in E(G)$, that is, $(h(u),0) = g(u,0) < g((u,v),0) = (e(u,v),0)$ and $(e(u,v),1) = g((u,v),1) < g(v,1) = (h(v),1)$. Therefore, $e(u,v) \in E(H)$ equals $(h(u),h(v))$. We conclude that $h$ is a homomorphism of $G$ to $H$.
\end{proof}

\begin{proposition}
\label{core}
Let $G$ be a graph. Then $P_G$ is a core if and only if $G$ is a core.
\end{proposition}
\begin{proof}
If $P_G$ is a core, then it is not homomorphic to any of its proper $k$-subposets. In particular, by Proposition~\ref{GHhomom}, there is no proper subgraph $H$ of $G$ such that $P_G$ is homomorphic to $P_H$. Thus, $G$ does not retract to any proper subgraph, and hence $G$ is a core.

If $P_G$ is not a core, then there is a homomorphism $h \colon P_G \to P'$ for some proper $k$-subposet $P' = \range h$ of $P_G$. It is clear from the proof of Proposition \ref{GHhomom} that the homomorphic image $P'$ of $P_G$ is of the form $P_H$ for some graph $H$. Then $H$ is a proper subgraph and a retract of $G$, and so $G$ is not a core.
\end{proof}

We describe a variant of the above representation of directed graphs by labeled posets. We associate with each directed graph $G$ the $3$-poset $L_G$, which is defined like $P_G$ but with a greatest element and a least element adjoined. The two new elements have label $2$. (For the empty graph $\emptyset$, we agree that $L_\emptyset$ is the empty $3$-poset.) It is easy to see that $L_G$ is a $3$-lattice if and only if $G$ is loopless. (A single loop gives rise to the $3$-poset shown in Figure \ref{fig:loop}, which is not a $3$-lattice.)
\begin{figure}
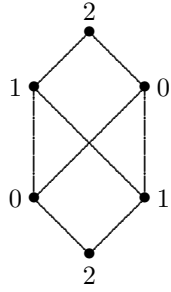

\[
\begindc{\undigraph}[3]
\obj(0,0)[a]{2}[\south]
\obj(-7,7)[b]{0}[\west]
\obj(7,7)[c]{1}[\east]
\obj(-7,21)[d]{1}[\west]
\obj(7,21)[e]{0}[\east]
\obj(0,28)[f]{2}[\north]
\mor{a}{b}{}
\mor{a}{c}{}
\mor{b}{d}{}
\mor{b}{e}{}
\mor{c}{d}{}
\mor{c}{e}{}
\mor{d}{f}{}
\mor{e}{f}{}
\enddc
\]
\caption{The $3$-poset representation of a loop.}
\label{fig:loop}
\end{figure}

\begin{proposition}
\label{GHhomomLattice}
Let $G$ and $H$ be directed graphs. Then $G$ is homomorphic to $H$ if and only if $L_G$ is homomorphic to $L_H$.
\end{proposition}
\begin{proof}
The proof is similar to that of Proposition \ref{GHhomom}. We only need to observe that the greatest and least elements are the only elements with label $2$, and every homomorphism must map the greatest and least elements to the greatest and least elements, respectively. Otherwise homomorphisms act as described in the proof of Proposition~\ref{GHhomom}.
\end{proof}

\begin{proposition}
Let $G$ be a graph. Then $L_G$ is a core if and only if $G$ is a core.
\end{proposition}
\begin{proof}
The proof is similar to that of Proposition \ref{core}.
\end{proof}

A countable poset is \emph{universal} if every countable poset can be embedded into it. We established in \cite{Universal} that the posets $\tilde{\mathcal{P}}_k$ ($k \geq 2$) and $\tilde{\mathcal{L}}_k$ ($k \geq 3$) are universal. Our representation of directed graphs by $2$-posets and that of loopless directed graphs by $3$-lattices provides a new proof of this fact.

\begin{theorem}
The posets $\tilde{\mathcal{P}}_k$ ($k \geq 2$) and $\tilde{\mathcal{L}}_k$ ($k \geq 3$) are universal.
\end{theorem}

\begin{proof}
It is a well-known fact that the homomorphism order of (loopless) directed graphs is universal (see \cite{PT}; see also Hubi\v{c}ka and Ne\v{s}et\v{r}il's \cite{HubN} simpler proof). The claim then follows from Propositions \ref{GHhomom} and \ref{GHhomomLattice}.
\end{proof}

How hard is it to find homomorphisms between $k$-posets? The $k$-poset representation of directed graphs given above allows us to transfer some complexity results from directed graphs to $k$-posets. We represent directed graphs by $k$-posets in such a way that there is a homomorphism between graphs if and only if there is homomorphism between their $k$-poset representations. This representation enables us to prove the NP-completeness of certain decision problems related to $k$-posets. More precisely, we will show that the problem of deciding whether there exists a homomorphism between two $k$-posets and the problem of deciding whether a $k$-poset admits a nonsurjective endomorphism are NP-complete. 

 We define the \emph{$k$-poset homomorphism problem} $\kHOM$ and the \emph{$k$-poset non-coreness problem} $\kCORE$ as follows. Note that cores are precisely the $k$-posets for which the answer to the question of $\kCORE$ is \textsc{no}.
\begin{quotation}
\noindent\textbf{Problem} $\kHOM$

\noindent\textit{Instance:} $k$-posets $(P,c)$ and $(P',c')$.

\noindent\textit{Question:} Is there a homomorphism $(P,c) \to (P',c')$?
\end{quotation}
\begin{quotation}
\noindent\textbf{Problem} $\kCORE$

\noindent\textit{Instance:} A $k$-poset $(P,c)$.

\noindent\textit{Question:} Is there a nonsurjective endomorphism of $(P,c)$?
\end{quotation}

\textbf{$k$-HOM} and \textbf{$k$-CORE} are analogues of the \emph{graph homomorphicity problem} $\HOM$ and the \emph{graph non-coreness problem} $\CORE$, defined as follows.
\begin{quotation}
\noindent\textbf{Problem} $\HOM$

\noindent\textit{Instance:} Graphs $G$ and $G'$.

\noindent\textit{Question:} Is there a homomorphism $G \to G'$?
\end{quotation}
\begin{quotation}
\noindent\textbf{Problem} $\CORE$

\noindent\textit{Instance:} A graph $G$.

\noindent\textit{Question:} Is there a nonsurjective homomorphism $G \to G$?
\end{quotation}

It is an easy exercise to show that both $\kHOM$ and $\kCORE$ are NP-complete, using the well-known fact that $\HOM$ and $\CORE$ are NP-complete \cite{HN1990,HN1992} and the representation of graphs by labeled posets as described in Section \ref{sec:representation}.

It follows from Proposition \ref{GHhomomLattice} that $\kHOM$ remains NP-complete even when its inputs are restricted to $k$-lattices for $k \geq 3$. However, for $2$-lattices, $\kHOM$ is solvable in polynomial time. For, it was shown by Kosub and Wagner \cite{KW} that every $2$-lattice is homomorphically equivalent to its longest alternating chain. For each $n \geq 1$, there are exactly two nonisomorphic alternating $2$-chains of cardinality $n$, and these are fully described by the length $n$ and the label of the least element. Thus a $2$-lattice $(L,c)$ is homomorphic to a $2$-lattice $(L',c')$ if and only if $\Alt(L,c) < \Alt(L',c')$ or $\Alt(L,c) = \Alt(L',c')$ and the least elements of $(L,c)$ and $(L',c')$ have the same label. It is an easy exercise to show that the alternation number of a finite $k$-poset can be determined in polynomial time.

Consider also the \emph{$k$-poset $(Q,d)$-homomorphicity problem} $\kQdHOM$, defined as follows. Here $(H,d)$ is a fixed $k$-poset and we should decide whether a given $k$-poset is homomorphic to $(H,d)$.
\begin{quotation}
\noindent\textbf{Problem} $\kQdHOM$

\noindent\textit{Instance:} A $k$-poset $(P,c)$.

\noindent\textit{Question:} Is there a homomorphism $(P,c) \to (Q,d)$?
\end{quotation}
This is an analogue of the \emph{graph $H$-colouring problem} $\HHOM$, defined as follows.
\begin{quotation}
\noindent\textbf{Problem} $\HHOM$

\noindent\textit{Instance:} A graph $G$.

\noindent\textit{Question:} Is there a homomorphism $G \to H$?
\end{quotation}

It is clear that $\kQdHOM$ is in NP for any $k$-poset $(Q,d)$. It was shown by Hell and Ne\v{s}et\v{r}il \cite{HN1990} that $\HHOM$ is NP-complete for any non-bipartite graph $H$, and it is polynomial-time solvable for any bipartite graph $H$. Thus, there are NP-complete cases of $\kQdHOM$, e.g., the cases where $(Q,d) = P_G$ for some nonbipartite graph $G$. There are also polynomial-time solvable cases, e.g., the cases where the labeling $d$ in $(Q,d)$ is a constant function---it only suffices to check whether the labeling $c$ of the input $(P,c)$ is constant function taking on the same value as $d$, and this can certainly be decided in polynomial time.

It remains an open question whether there is a dichotomy between the poly\-no\-mi\-al-time solvable and NP-complete cases of $\kQdHOM$, analogously to that of $\HHOM$.

\section{Properties of the homomorphism order of $k$-posets}
The homomorphism order of $k$-posets forms a distributive lattice with disjoint union as join, and label-matching product as meet~\cite{Universal}. 
The \emph{disjoint union} of a family $(S_i)_{i \in I}$ of sets is defined as the set
\[
\bigdotcup_{i \in I} S_i = \{(i,x) \mid i \in I,\, x \in S_i\}.
\]
If $I = \{1,2\}$, then we write $S_1 \dotcup S_2$ for $\bigdotcup_{i \in \{1,2\}} S_i$. The \emph{disjoint union} of a family $(P_i,c_i)_{i \in I}$ of $k$-posets is defined to be the $k$-poset $\bigdotcup_{i \in I} (P_i,c_i) = (\bigdotcup_{i \in I} P_i, d)$, where $d(i,x) = c_i(x)$ for all $(i,x) \in \bigdotcup_{i \in I} P_i$, and the order on $\bigdotcup_{i \in I} P_i$ is defined as $(i,x) \leq (j,y)$ if and only if $i = j$ and $x \leq y$ in $P_i$.

The \emph{label-matching product} of a family $(P_i, c_i)_{i \in I}$ of $k$-posets is defined to be the $k$-poset $\bigotimes_{i \in I} (P_i, c_i) := (Q,d)$, where
\[
Q := \{(a_i)_{i \in I} \in \prod_{i \in I} P_i \mid \text{$c_i(a_i) = c_j(a_j)$ for all $i, j \in I$}\},
\]
$(a_i)_{i \in I} \leq (b_i)_{i \in I}$ in $Q$ if and only if $a_i \leq b_i$ in $P_i$ for all $i \in I$, and the labeling is defined by $d((a_i)_{i \in I}) = c_i(a_i)$ for some $i \in I$ (the choice of $i$ does not matter by the definition of $Q$). If $I = \{1, 2\}$, then we write $(P_1, c_1) \otimes (P_2, c_2)$ for $\bigotimes_{i \in \{1, 2\}} (P_i, c_i)$.

It was shown in \cite{Universal} that $(\tilde{\mathcal{P}}_k, {\leq})$ is a distributive lattice with the lattice operations defined as follows:
\[
(P,c) \vee (P',c') = (P,c) \dotcup (P',c'),\quad\text{ and }\quad
(P,c) \wedge (P',c') = (P,c) \otimes (P',c').
\]
 Here the lattice operations are defined in terms of equivalence class representatives.

\begin{proposition}
The join-irreducible elements of $(\tilde{\mathcal{P}}_k, {\leq})$ are (the equivalence classes of) the cores with at most one connected component.
\end{proposition}
\begin{proof}
The empty $k$-poset is the smallest element of $\tilde{\mathcal{P}}_k$, so it is clearly join-irreducible. We can then assume that $(P,c)$ is a nonempty core. Let $(P_1,c_1), \dotsc,\linebreak[0] (P_n,c_n)$ be the connected components of $(P,c)$. These connected component are cores and they are pairwise incomparable under $\leq$. If $n > 1$, then $(P,c)$ is the disjoint union of its connected components and thus it is not join-irreducible.

Assume then that $n = 1$. Suppose, on the contrary, that $(P,c)$ is not join-irreducible. Then there exist cores $(Q_1,d_1)$ and $(Q_2,d_2)$ that are not equivalent to $(P,c)$ such that $(P,c) \equiv (Q_1,d_1) \dotcup (Q_2,d_2)$. Thus there exist homomorphisms $h \colon (P,c) \to (Q_1,d_1) \dotcup (Q_2,d_2)$ and $g \colon (Q_1,d_1) \dotcup (Q_2,d_2) \to (P,d)$. Since $(P,c)$ is connected, $h$ is in fact a homomorphism of $(P,c)$ to $(Q_1,d_1)$ or to $(Q_2,d_2)$. Furthermore, for $i = 1,2$, the restriction of $g$ to $Q_i$ is a homomorphism of $(Q_i,d_i)$ to $(P,c)$. Thus, $(P,c)$ is homomorphically equivalent to either $(Q_1,d_1)$ or $(Q_2,d_2)$, a contradiction.
\end{proof}

Denote by $\mathcal{J}_k$ the set of join-irreducible elements of the lattice $(\tilde{\mathcal{P}}_k, {\leq})$, which we just showed to be the set of cores with at most one connected component. Since every finite core has only a finite number of connected components and is the supremum of its connected components, we conclude that every element of $\tilde{\mathcal{P}}_k$ is the join of a finite number of elements of $\mathcal{J}_k$. Hence $\mathcal{J}_k$ is a join-dense subset of $\tilde{\mathcal{P}}_k$. As we have mentioned already, $\tilde{\mathcal{P}}_k$ is not complete. The smallest complete poset (lattice) containing $\tilde{\mathcal{P}}_k$ is its Dedekind-MacNeille completion.  One way to construct it is to take the set of normal ideals of $\tilde{\mathcal{P}}_k$ ordered by inclusion \cite{MacNeille} or to take the concept lattices of the formal contexts $\bigl(\tilde{\mathcal{P}}_k,\tilde{\mathcal{P}}_k,\leq\bigr)$ or $\bigl(\mathcal{J}_k,\tilde{\mathcal{P}}_k,\leq\bigr)$~\cite{GW99}. We denote by $\hat{\mathcal{P}}_k$ the Dedekind-MacNeille completion of $\tilde{\mathcal{P}}_k$. Note that $\tilde{\mathcal{P}}_k$ is join-dense and meet-dense in $\hat{\mathcal{P}}_k$. Then $\mathcal{J}_k$ is a join-dense subset of $\hat{\mathcal{P}}_k$. 
{\em Is $\hat{\mathcal{P}}_k$ an algebraic lattice?} 
Recall that an element $a$ of a complete lattice $L$ is called \emph{compact} if $a\leq{\bigvee}X$ for some $X\subseteq L$ implies that $a\leq{\bigvee}X_1$ for some finite $X_1\subseteq X$, and that a complete lattice $L$ is called \emph{algebraic} or \emph{compactly generated} if every element is the join of compact elements. More generally, {\em is the MacNeille completion of any compactly generated lattice also compactly generated?} Before we answer this question, we first investigate a subposet of $\hat{\mathcal{P}}_k$ in which we can compute all suprema and infima of elements of $\tilde{\mathcal{P}}_k$.

We are looking for posets containing $\tilde{\mathcal{P}}_k$ as subposet in which we can compute all suprema and infima of elements of $\tilde{\mathcal{P}}_k$. Since $\tilde{\mathcal{P}}_k$ is countably infinite, each completion should contain at least the countable unions of finite $k$-posets. Since any countable union of finite sets is again countable, we will start by enlarging a bit the class $\tilde{\mathcal{P}}_k$. We denote by $\mathcal{P}_{k\omega}$ the class of countable $k$-posets. The homomorphism quasi-order on $\mathcal{P}_{k\omega}$ is defined in the same way as for finite $k$-posets and it induces a partial order on the quotient $\mathcal{P}_{k\omega} / {\equiv}$, which we will denote by $\tilde{\mathcal{P}}_{k \omega}$. A poset $(P,\leq)$ is called \emph{$\omega$-complete}\footnote{This notion can be generalized to $\kappa$-completeness for any cardinal $\kappa$ as follows: a poset $(P,\leq)$ is $\kappa$-complete if the suprema and infima of subsets of cardinality at most $\kappa$ exist in $P$.} if the suprema and infima of countable subsets of $P$ exist. For countable posets, completeness and $\omega$-completeness coincide. 

\begin{lemma}
The poset $(P_{k\omega},\le)$ is $\omega$-complete.
\end{lemma}

\begin{proof}
Suprema and infima will be constructed as in \cite{Universal}. 
Let $(P_t,c_t)_{t\in T}$ be a countable family of elements of $\mathcal{P}_{k\omega}$. Define a $k$-poset $(\bar{P},c)$ as the disjoint unions of $(P_t,c_t)$'s, i.e.,
\[
\bar{P} := \bigdotcup_{t\in T} P_t
\quad\text{and}\quad
c(t,a) = c_t(a).
\]
Then $\bar{P}$ is countable and $(\bar{P},c)$ is in $\mathcal{P}_{k\omega}$. Moreover $(\bar{P},c)$ is the supremum of $(P_t,c_t)_{t \in T}$. In fact, it is clear that each inclusion map $\tau_t \colon P_t \to \bar{P}$, $x \mapsto (t,x)$ is a homomorphism of $k$-posets; if $(P_t,c_t) \leq (Q,d)$, then there are $k$-poset homomorphisms $h_t \colon P_t \to Q$ for each $t \in T$; define $h \colon \bar{P} \to Q$ by $h(t,p) := h_t(p)$, for every $t \in T$ and $p \in P_t$. The mapping $h$ is a $k$-poset homomorphism and thus $(\bar{P},c) \leq (Q,d)$. Therefore $(\bar{P},c)$ is the supremum of $(P_t,c_t)_{t \in T}$. For the infimum, consider the label-matching product $(\tilde{P},\tilde{c})$ of $\bigl((P_t,c_t)\bigr)_{t \in T}$ given by:
\[
\tilde{P} := \{a \in \prod_{t \in T} P_t \mid \text{$c_t(a_t) = c_s(a_s)$ for all $s, t \in T$}\}
\quad \text{ and } \quad
\tilde{c}(a) := c_t(a_t).
\]
$\tilde{P}$ keeps only the elements having the same label on all components and sets this as its label. 
Of course the projections $\pi_t \colon (\tilde{P},\tilde{c}) \to (P_t,c_t)$, $a \mapsto a_t$ ($t \in T$) are $k$-poset homomorphisms; thus $(\tilde{P},\tilde{c}) \leq (P_t,c_t)$ for all $t \in T$. If $(Q,d) \leq (P_t,c_t)$ for all $t \in T$, then there are $k$-poset homomorphisms $g_t \colon (Q,d) \to (P_t,c_t)$. Define $g \colon Q \to \tilde{P}$ by $g(q) := \bigl( g_t(q) \bigr)_{t \in T}$. Then $g$ is a homomorphism of $k$-posets, and $(Q,d) \leq (\tilde{P},\tilde{c})$.
\end{proof}

As an $\omega$-complete poset, $(\tilde{\mathcal{P}}_{k\omega},\leq)$ is a lattice containing $(\tilde{\mathcal{P}}_k,\leq)$ as a sublattice, in which all suprema and infima of $\tilde{\mathcal{P}}_k$ exist. An $\omega$-complete poset $(P,\leq)$ is called \emph{$\omega$-join-distributive} (\emph{$\omega$-meet-distributive}) if for any index set $T$ of cardinality at most $\omega$, for any family $(a_t)_{t \in T}$ of elements of $P$ and for any $b \in P$, we have 
\begin{align*}
b \wedge \bigvee_{t \in T} a_t &= \bigvee_{t \in T}(b \wedge a_t) \\
\text{(}
\quad
b \vee \bigwedge_{t \in T} a_t &= \bigwedge_{t \in T}(b \vee a_t),
\qquad
\text{respectively).}
\end{align*} 
If an $\omega$-complete poset is both $\omega$-join- and $\omega$-meet-distributive, we call it \emph{$\omega$-distri\-bu\-ti\-ve}
\footnote{Replacing $\omega$ with $\kappa$ gives $\kappa$-distributivity. This is a generalization of distributivity ($\kappa=2$). For finite cardinals $\kappa$, the notions of $\kappa$-join-distributivity, $\kappa$-meet-distributivity and distributivity are equivalent. This is unfortunately no longer true for $\kappa\geq\omega$.}.
The $\omega$-complete poset $(\tilde{\mathcal{P}}_{k\omega}, {\leq})$ is $\omega$-distributive as we can see from Lemmas~\ref{L:Inf-distr} and~\ref{L:join-inf-distr}.

\begin{lemma}\label{L:Inf-distr}
The $\omega$-complete poset $(\tilde{\mathcal{P}}_{k\omega}, {\leq})$ is $\omega$-join-distributive.
\end{lemma}

\begin{proof}
Let $b := (Q,d) \in \tilde{\mathcal{P}}_{k \omega}$ and $(P_t,c_t)_{t \in T}$ be a countable family of elements of $\tilde{\mathcal{P}}_{k\omega}$. We set $a_t := (P_t,c_t)$. To show that $(\tilde{\mathcal{P}}_{k\omega}, {\leq})$ is $\omega$-join-distributive, we observe that $(Q,d) \otimes \left(\bigdotcup_{t\in T}(P_t,c_t)\right)$ and $\bigdotcup_{t\in T} \bigl( (Q,d) \otimes (P_t,c_t) \bigr)$ are homomorphically equivalent. In fact for any $t$, $x$ and $y$, we have
\begin{eqnarray*}
(x,t,y)\in (Q,d) \otimes \bigl(\bigdotcup_{t \in T}(P_t,c_t) \bigr)
& \iff & x \in Q,\ t \in T,\ a \in P_t\ \text{ and}\\
& \phantom{\iff} & d(x) = \bar{c}(t,y) = c_t(y) \\
& \iff & (x,y) \in (Q,d) \otimes (P_t,c_t) \\
& \iff & (t,x,y) \in \bigdotcup_{t \in T} (Q,d) \otimes (P_t,c_t);
\end{eqnarray*}
then $h \colon (x,t,y) \mapsto (t,x,y)$ defines a $k$-poset isomorphism of $(Q,d) \otimes \bigl( \bigdotcup_{t \in T} (P_t,c_t) \bigr)$ onto $\bigdotcup_{t \in T} \bigl( (Q,d) \otimes (P_t,c_t) \bigr)$. Note that the label of $(x,t,y)$ in $(Q,d) \otimes \bigl( \bigdotcup_{t \in T} (P_t,c_t) \bigr)$ is $c_t(y)$, which is also the label of $(t,x,y)$ in $\bigdotcup_{t \in T} (Q,d) \otimes (P_t,c_t)$. Thus in $(\mathcal{P}_{k\omega}, {\leq})$ we have 
\[
b \wedge \bigvee_{t \in T} a_t = (Q,d) \otimes \bigl( \bigdotcup_{t \in T} (P_t,c_t) \bigr) = \bigdotcup_{t \in T} \bigl( (Q,d) \otimes (P_t,c_t) \bigr) = \bigvee_{t \in T}(b \wedge a_t).
\] 
\end{proof}

\begin{lemma}\label{L:join-inf-distr}
The $\omega$-complete poset $(\tilde{\mathcal{P}}_{k\omega},\leq)$ is $\omega$-meet-distributive.
\end{lemma}

\begin{proof}
We know that 
\[
b \vee \bigwedge_{t \in T} a_t \leq \bigwedge_{t \in T} (b \vee a_t)
\] 
always holds. Our aim is to find a $k$-poset homomorphism of 
$\bigotimes_{t \in T} \bigl( (Q,d) \dotcup (P_t, c_t) \bigr)$ to 
$(Q,d) \dotcup \bigotimes_{t \in T} (P_t, c_t)$.
Note that 
\[
(s,x) \in (Q,d) \dotcup \bigotimes_{t \in T}(P_t,c_t)
\iff
s=1\ \&\ x \in Q \text{ or } s=2\ \&\ x \in \bigotimes_{t \in T}(P_t,c_t).
\]
Now let $\mathfrak{X} \in \bigotimes_{t \in T} \bigl( (Q,d) \dotcup (P_t,c_t) \bigr)$. Then $\mathfrak{X}$ is a $T$-sequence of elements of $(Q,d) \dotcup (P_t,c_t)$ whose components have the same label, say $\mathfrak{X} = (i_t, x_t)_{t \in T}$ with $i_t \in \{1, 2\}$ and $x_t \in Q$ if $i_t = 1$ and $x_t \in P_t$ if $i_t = 2$, and $(d \dotcup c_t)(i_t,x_t) = (d \dotcup c_s)(i_s,x_s)$ for all $s, t \in T$. 
Define the map $h \colon \bigotimes_{t \in T} \bigl( (Q,d) \dotcup (P_t,c_t) \bigr) \to (Q,d) \dotcup \bigotimes_{t \in T} (P_t, c_t)$ as follows:
\[
h((i_t, x_t)_{t \in T}) =
\begin{cases}
(2, (x_t)_{t \in T}) & \text{if $i_t = 2$ for all $t \in T$,}\\
(1, x_j) & \text{if $S = \{t \in T \mid i_t = 1\} \neq \emptyset$ and $j = \min S$.}
\end{cases}
\]
(For an arbitrary cardinality $\kappa$, we assume that $T$ is well-ordered, and we take the minimum with respect to a fixed well-ordering.)
We need to verify that $h$ is a homomorphism. It is clear that $h$ preserves labels. As regards preservation of order, let $\mathfrak{X}_\ell = (i_t^\ell, x_t^\ell)_{t \in T}$ ($\ell = 1, 2$), and assume that $\mathfrak{X}_1 \leq \mathfrak{X}_2$ in $\bigotimes_{t \in T} \bigl( (Q,d) \dotcup (P_t,c_t) \bigr)$. Then $(i_t^1, x_t^1) \leq (i_t^2, x_t^2)$ in $(Q,d) \dotcup (P_t,c_t)$ for all $t \in T$, which in turn implies that $i_t^1 = i_t^2$ and $x_t^1 \leq x_t^2$ (in $(Q,d)$ or in $(P_t,c_t)$, depending on the value of $i_t^1$) for all $t \in T$. Thus the sets
\[
S_\ell = \{t \in T \mid i_t^\ell = 1\}
\quad (\ell = 1, 2)
\]
are equal. Hence either $h(\mathfrak{X}_\ell) = (2, (x_t^\ell)_{t \in T})$ for $\ell = 1, 2$ or $h(\mathfrak{X}_\ell) = (1, x_j^\ell)$ for $\ell = 1, 2$, where $j = \min S_1 = \min S_2$. In both cases it is obvious that $h(\mathfrak{X}_1) \leq h(\mathfrak{X}_2)$.
\end{proof}

\begin{theorem}
\label{T:Inf-distr}
Let $(a_t)_{t \in T}$ be a family of elements of $\tilde{\mathcal{P}}_k$, and let $b \in \tilde{\mathcal{P}}_k$. If $(a_t)_{t \in T}$ has a supremum in $\tilde{\mathcal{P}}_k$, then the family $(b \wedge a_t)_{t \in T}$ has a supremum in $\tilde{\mathcal{P}}_k$, and it holds that
\[
b \wedge \bigvee_{t \in T} a_t = \bigvee_{t \in T} (b \wedge a_t).
\] 
Similarly, if $(a_t)_{t \in T}$ has an infimum in $\tilde{\mathcal{P}}_k$, then the family $(b \wedge a_t)_{t \in T}$ has an infimum in $\tilde{\mathcal{P}}_k$, and it holds that
\[
b \vee \bigwedge_{t \in T} a_t = \bigwedge_{t \in T} (b \vee a_t).
\] 
\end{theorem}
\begin{proof}
The claim follows from Lemmas~\ref{L:Inf-distr} and \ref{L:join-inf-distr} and the fact that we are dealing with finite $k$-posets only.
\end{proof}

\begin{corollary}
$(\tilde{\mathcal{P}}_{k\omega}, {\leq})$ is a distributive lattice.
\end{corollary}

\begin{proposition}
The cores with at most one connected component are compact and prime elements of $(\tilde{\mathcal{P}}_{k\omega}, {\leq})$.
\end{proposition}

\begin{proof}
Let $a\in\mathcal{J}_k$ and $X\subseteq \tilde{\mathcal{P}}_{k\omega}$ such that $a\leq{\bigvee}X$. As $\tilde{P}_k$ is countable and join-dense in $\tilde{\mathcal{P}}_{k\omega}$, we can assume that $X$ is countable. We are looking for a finite subset $X_1\subseteq X$ such that $a\leq{\bigvee}X_1$. We have $a=a\wedge{\bigvee}X=\bigvee\{a\wedge x\mid x\in X\}$, by the $\omega$-join-distributivity. Therefore there is a $k$-poset homomorphism $\varphi: a\to \bigdotcup\{a\otimes x\mid x\in X\}$. Since $a$ is connected, $\varphi(a)$ is also connected and there is an $x_0\in X$ such that $\varphi(a)\subseteq a\otimes x_0$. Thus $\varphi$ is a $k$-poset homomorphism from $a$ to $a\otimes x_0$, i.e., $a\leq a\wedge x_0$. Therefore we can let $X_1:=\{x_0\}\subseteq X$. 
\end{proof}

All elements of $\tilde{\mathcal{P}}_k$ are finite joins of elements of $\mathcal{J}_k$, and are hence compact in $\tilde{\mathcal{P}}_{k\omega}$. Are they also compact in the MacNeille completion $\hat{\mathcal{P}}_k$ of $\tilde{\mathcal{P}}_k$? This is still an open question, and seems to be intimately related with the distributivity of $\hat{\mathcal{P}}_k$. A positive answer will say that $\hat{\mathcal{P}}_k$ is an algebraic lattice.

\section{Bounded $k$-posets with fixed labels at the extreme points}

Recall that we denote by $\mathcal{L}_k$ the set of all $k$-lattices and we denote $\tilde{\mathcal{L}}_k = \mathcal{L}_k / {\equiv}$. $\tilde{\mathcal{L}}_k$ is clearly a subposet of $\tilde{\mathcal{P}}_k$, but it is not a sublattice of $\tilde{\mathcal{P}}_k$, for the simple reason that the disjoint union of two incomparable $k$-lattices is not (homomorphically equivalent to) a $k$-lattice. Even if we consider the subposet of $\tilde{\mathcal{P}}_k$ consisting of (the equivalence classes of) those $k$-posets whose connected components are lattices, we do not have a sublattice nor even a meet-subsemilattice of $\tilde{\mathcal{P}}_k$. This is due to the fact that the label-matching product of two $k$-lattices is generally not (homomorphically equivalent to) a $k$-lattice, as Figure \ref{fig:prodoflattices} illustrates. An identical argument shows that $k$-trees do not constitute a sublattice of $\tilde{\mathcal{P}}_k$, and neither do $k$-forests ($k$-posets whose connected components are $k$-trees).

\begin{figure}
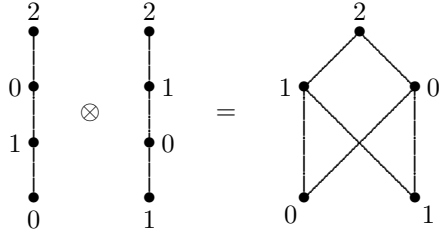

\[
\begindc{\undigraph}[3]
\obj(0,-10)[a]{0}[\south]
\obj(0,-3)[b]{1}[\west]
\obj(0,4)[c]{0}[\west]
\obj(0,11)[d]{2}[\north]
\mor{a}{b}{}
\mor{b}{c}{}
\mor{c}{d}{}
\enddc
\quad\;\otimes\quad
\begindc{\undigraph}[3]
\obj(0,-10)[a]{1}[\south]
\obj(0,-3)[b]{0}[\east]
\obj(0,4)[c]{1}[\east]
\obj(0,11)[d]{2}[\north]
\mor{a}{b}{}
\mor{b}{c}{}
\mor{c}{d}{}
\enddc
\quad\;=\quad
\begindc{\undigraph}[3]
\obj(-7,-10)[b]{0}[\southwest]
\obj(7,-10)[c]{1}[\southeast]
\obj(-7,4)[d]{1}[\west]
\obj(7,4)[e]{0}[\east]
\obj(0,11)[f]{2}[\north]
\mor{b}{d}{}
\mor{b}{e}{}
\mor{c}{d}{}
\mor{c}{e}{}
\mor{d}{f}{}
\mor{e}{f}{}
\enddc
\]
\caption{The label-matching product of $k$-lattices is not generally a $k$-lattice.}
\label{fig:prodoflattices}
\end{figure}

In this section, we will consider families of bounded $k$-posets with fixed labels on their extreme points. These families constitute meet-subsemilattices of $\tilde{\mathcal{P}}_k$. We will describe the suprema within these families, and we establish that these families constitute universal distributive lattices under the homomorphism order.

Let $k \geq 1$, and let $a, b \in \{0, 1, \dotsc, k-1\}$. Denote by $\mathcal{P}_k^{ab}$ the set of finite bounded $k$-posets $(P,c)$ with a largest element $\top$ and a smallest element $\bot$ such that $c(\top) = a$ and $c(\bot) = b$. Denote $P^\dagger := P \setminus \{\top,\bot\}$. Again, denote by $\tilde{\mathcal{P}}_k^{ab}$ the quotient $\mathcal{P}_k^{ab} / {\equiv}$.

Let $(P,c), (P',c') \in \mathcal{P}_k^{ab}$. It is easy to verify that the label-matching product $(P,c) \otimes (P',c')$ is again in $\mathcal{P}_k^{ab}$, and hence $\tilde{\mathcal{P}}_k^{ab}$ is a meet-subsemilattice of $\tilde{\mathcal{P}}_k$. However, the core of the disjoint union $(P,c) \dotcup (P',c')$ is generally not a bounded $k$-poset, and hence we need to verify if $(P,c)$ and $(P',c')$ have an infimum in $\tilde{\mathcal{P}}_k^{ab}$.

Define the binary operation $\uplus$ on $\mathcal{P}_k^{ab}$ as follows. For $i = 1, 2$, let $(P_i,c_i) \in \mathcal{P}_k^{ab}$, and let $\top_{P_i}$ and $\bot_{P_i}$ be the largest and smallest elements of $P_i$. We let $(P_1,c_1) \uplus (P_2,c_2) = (Q,d)$, where
\[
Q = (P_1^\dagger \dotcup P_2^\dagger) \cup \{\top_Q, \bot_Q\}
\]
where $\top_Q, \bot_Q$ are new elements not occurring in $P_1$ nor $P_2$. The ordering of $Q$ is defined as follows: $\top_Q$ and $\bot_Q$ are the largest and the smallest element of $Q$, respectively, and for $(i,x), (j,y) \in P_1^\dagger \dotcup P_2^\dagger$, we have $(i,x) \leq (j,y)$ if and only if $i = j$ and $x \leq y$ in $P_i$. The labeling $d$ of $Q$ is defined by
\[
d(x) =
\begin{cases}
a & \text{if $x = \top_Q$,} \\
b & \text{if $x = \bot_Q$,} \\
c_i(y) & \text{if $x = (i,y) \in P_1^\dagger \dotcup P_2^\dagger$.}
\end{cases}
\]
Thus, we can think of $(P_1, c_1) \uplus (P_2, c_2)$ being obtained from the disjoint union $(P_1, c_1) \dotcup (P_2, c_2)$ by gluing together the top and bottom elements of the connected components.

\begin{lemma}
\label{L:supPkab}
$(P_1,c_1) \uplus (P_2,c_2)$ is the supremum of $(P_1,c_1)$ and $(P_2,c_2)$ in $\tilde{\mathcal{P}}_k^{ab}$.
\end{lemma}
\begin{proof}
Denote $(Q,d) = (P_1,c_1) \uplus (P_2,c_2)$ For $i = 1, 2$, the mapping $h_i \colon (P_i,c_i) \to (Q,d)$ given by
\[
h_i(x) =
\begin{cases}
\top_Q & \text{if $x = \top_{P_i}$,} \\
\bot_Q & \text{if $x = \bot_{P_i}$,} \\
(i,x) & \text{if $x \in P_i^\dagger$}
\end{cases}
\]
is easily seen to be a homomorphism.

Now, assume that $(P',c') \in \mathcal{P}_k^{ab}$ and there exist homomorphisms $h_i \colon (P_i, c_i) \to (P',c')$ for $i = 1, 2$. Define a map $h \colon (Q,d) \to (P',c')$ by
\[
h(x) =
\begin{cases}
\top_{P'} & \text{if $x = \top_Q$,} \\
\bot_{P'} & \text{if $x = \bot_Q$,} \\
h_i(y) & \text{if $x = (i,y) \in Q^\dagger$.}
\end{cases}
\]
It is straightforward to verify that $h$ is a homomorphism. We conclude that $(P_1,c_1) \uplus (P_2,c_2)$ is the supremum of $(P_1,c_1)$ and $(P_2,c_2)$ in $\tilde{\mathcal{P}}_k^{ab}$.
\end{proof}

\begin{proposition}
$(\tilde{\mathcal{P}}_k^{ab}; \otimes, \uplus)$ is a distributive lattice.
\end{proposition}
\begin{proof}
The claim that $(\tilde{\mathcal{P}}_k^{ab}; \otimes, \uplus)$ is a lattice follows from Lemma \ref{L:supPkab} and the discussion preceding it.

Let $(P_i,c_i) \in \mathcal{P}_k^{ab}$ for $i = 1, 2, 3$. We will verify that the distributive law
\[
P_1 \otimes (P_2 \uplus P_3) \equiv (P_1 \otimes P_2) \uplus (P_1 \otimes P_3)
\]
holds by showing that the $k$-posets on each side of the above equation are homomorphically equivalent.

First, define the map $h \colon P_1 \otimes (P_2 \uplus P_3) \to (P_1 \otimes P_2) \uplus (P_1 \otimes P_3)$ by
\[
h(X,Y) =
\begin{cases}
\top & \text{if $X = \top_{P_1}$ or $Y = \top_{P_2 \uplus P_3}$,} \\
\bot & \text{if $X = \bot_{P_1}$ or $Y = \bot_{P_2 \uplus P_3}$,} \\
(i,(X,y)) & \text{if $X \in P^\dagger_1$, $Y = (i,y)$, $y \in P_{i+1}^\dagger$ ($i = 1, 2$).}
\end{cases}
\]
It is clear that $h$ is label-preserving. We need to verify that $h$ is also order-preserving. Thus, let $(X,Y) < (X',Y')$ in $P_1 \otimes (P_2 \uplus P_3)$. If $X = \bot_{P_1}$ or $Y = \bot_{P_2 \uplus P_3}$ or $X' = \top_{P_1}$ or $Y' = \top_{P_2 \uplus P_3}$, then it is clear that $h(X,Y) \leq h(X',Y')$. Otherwise $X, X' \in P_1^\dagger$, $Y, Y' \in (P_2 \uplus P_3)^\dagger$ and so $X \leq X'$ in $P_1$ and $Y \leq Y'$ in $P_2 \uplus P_3$. The latter condition implies that $Y = (i,y)$, $Y' = (i,y')$ for some $i \in \{1, 2\}$, $y, y' \in P_{i+1}$ and $y \leq y'$ in $P_{i+1}$. Thus,
\[
h(X,Y) = (i,(X,y)) \leq (i,(X',y')) = h(X',Y')
\quad\text{in $(P_1 \otimes P_2) \uplus (P_1 \otimes P_3)$.}
\]

Next, we define the map $g \colon (P_1 \otimes P_2) \uplus (P_1 \otimes P_3) \to P_1 \otimes (P_2 \uplus P_3)$ by
\[
g(X) =
\begin{cases}
(\top_{P_1}, \top_{P_2 \uplus P_3}) & \text{if $X = \top$,} \\
(\bot_{P_1}, \bot_{P_2 \uplus P_3}) & \text{if $X = \bot$,} \\
(x,(i,y)) & \text{if $X = (i, (x,y)) \in \bigl((P_1 \otimes P_2) \uplus (P_1 \otimes P_3)\bigr)^\dagger$.}
\end{cases}
\]
It is clear that $g$ is label-preserving. We need to verify that $g$ is also order-preserving. Thus, let $X < X'$ in $(P_1 \otimes P_2) \uplus (P_1 \otimes P_3)$. If $X = \bot$ or $Y = \top$, then it is clear that $g(X) \leq g(X')$. Otherwise $X, X' \in \bigl( (P_1 \otimes P_2) \uplus (P_1 \otimes P_3) \bigr)^\dagger$ and so $X = (i, (x,y))$, $X' = (i, (x', y'))$ for some $i \in \{1, 2\}$ and $x, x' \in P_1$, $y, y' \in P_{i+1}$ and $x \leq x'$ in $P_1$ and $y \leq y'$ in $P_{i+1}$. Thus
\[
h(X) = (x, (i,y)) \leq (x', (i, y')) = h(X')
\quad\text{in $P_1 \otimes (P_2 \uplus P_3)$.}
\]

Since both $h$ and $g$ are homomorphisms, we conclude that the claimed homomorphical equivalence holds.
\end{proof}

\begin{theorem}
The posets $\tilde{\mathcal{P}}_k^{ab}$ and $\tilde{\mathcal{L}}_k^{ab}$ are universal for every $k \geq 3$, $a, b \in \{0, \dotsc, k-1\}$.
\end{theorem}

\begin{proof}
The proof is a simple adaptation of the proof of the universality of $\tilde{\mathcal{L}}_k$ presented in \cite[Theorem 4.6]{Universal}. The $k$-posets $\mathcal{E}(A)$ used in the representation of an arbitrary countable poset are $3$-lattices. We just need to adjoin new top and bottom elements $\top$ and $\bot$ with labels $c(\top) = a$ and $c(\bot) = b$. The resulting $k$-posets $\mathcal{E}'(A)$ are members of $\tilde{\mathcal{L}}_k^{ab}$, and it is clear that there exists a homomorphism from $\mathcal{E}'(A)$ to $\mathcal{E}'(B)$ if and only if there exists a homomorphism from $\mathcal{E}(A)$ to $\mathcal{E}(B)$. The claim thus follows.
\end{proof}

\section*{Acknowledgements}

This work was initiated while the first author was visiting Tampere University of Technology, and some parts of it were carried out while both authors were visiting the Université du Québec en Outaouais and while the first author was visiting the University of Luxembourg. We would like to thank the above-mentioned universities for providing working facilities.

The authors would like to thank Ross Willard for helpful discussions of the topic.

\end{document}